\documentclass[11pt]{amsart}
\usepackage{amsmath,amsmath,latexsym,amssymb,amsfonts,amsbsy, amsthm}
\usepackage{mathrsfs}
\usepackage{fancyhdr}
\usepackage{latexsym}
\usepackage{multicol,graphics}
\usepackage{amssymb}
\usepackage{graphicx}
\usepackage{color}
\usepackage{bm}
\allowdisplaybreaks

\newcommand{\N}{{\mathbb N}}
\newcommand{\R}{{\mathbb R}}
\newcommand{\Z}{\mathbb{Z}}

\newtheorem{theorem}{Theorem}[section]
\newtheorem{lemma}[theorem]{Lemma}

\newtheorem{remark}[theorem]{Remark}
\newtheorem{corollary}[theorem]{Corollary}

\numberwithin{equation}{section}

\begin{document}

\title[]{Hexagonal standing wave patterns  of a two-dimensional Boussinesq system}
\author{Shenghao Li,  Min Chen, Bing-Yu Zhang}
\date{}
\begin{abstract}
We prove the existence of a large family of two-dimensional standing waves, that are triple periodic solutions, for a Boussinesq system which describes two-way propagation of water waves in a channel. Our proof uses the Lyapunov-Schmidt method to find the bifurcation standing waves.
\end{abstract}



\maketitle


\section{Introduction}
\setcounter{equation}{0}
\setcounter{theorem}{0}

In this article, we study the  standing wave patterns of a  two-dimensional system
\begin{equation}\label{2d}
\begin{cases}
\eta_t+\nabla\cdot\textbf{v}+ \nabla \cdot (\eta \textbf{v} )-\frac16\Delta \eta_t=0,\\
\textbf{v}_t+ \nabla \eta + \frac{1}{2} \nabla |\textbf{v}|^2-\frac16 \Delta \textbf{v}_t=0,
\end{cases}
\end{equation}
which was put forward by  Bona, Colin and Lannes \cite{95}. The horizontal location $\mathbf{x}=(x_1, x_2)$ and time $t$ are scaled by $h_0$ and $\sqrt{h_0/g}$ with $h_0$ being  the average
depth of water in the undisturbed state and  $g$ being the acceleration of gravity.
The unknown function $\eta(\mathbf{x},t)$ and $\mathbf{v}(\mathbf{x},t)=(v_1(\mathbf{x},t), v_2(\mathbf{x},t))$,
 represent the dimensionless deviation of the water surface from its undisturbed state and the horizontal velocity at the level of $\sqrt{2/3}h_0$ of the depth of the undisturbed fluid,
 scaled by $h_0$ and $\sqrt{gh_0}$,
 respectively.

 The system (\ref{2d}) and its one-dimensional version
 \begin{equation}\label{abcd}
\begin{cases}
&\eta_t+u_x+(u\eta)_x+au_{xxx}-bu_{xxt}=0,\\
&u_t+\eta_x+uu_x+c\eta_{xxx}-du_{xxt}=0,
\end{cases}
\end{equation}
  named as Boussineq-type systems or $abcd-$systems, was introduced by Bona, Chen and Saut \cite{7,6} describe the small-amplitude and long wavelength gravity waves of an ideal, incompressible liquid. The systems (\ref{2d}) and (\ref{abcd}) are the first order approximations to the full Euler equations in the small parameters $\epsilon_1=A/h_0$ and $\epsilon_2=h_0^2/\lambda^2$, where $A$ is the typical wave amplitude and $\lambda$ is the typical wavelength
 (see Bona, Colin and Lannes \cite{95} for the rigorous proof). For the one-dimensional systems, they  are at the same order of accuracy with
  the well-known KdV (Korteweg-de Vries), BBM (Benjamin-Bona-Mahoney) and Boussinesq equations, but works for a wider range of problems, meaning the wave do not have to be traveling in one direction.

 The water wave patterns such as traveling waves and standing waves have been studied in different modelings in the past (c.f \cite{97,96,98,99,100,101,103,102}).
  In this paper, we aim to continue the study of Chen and Iooss on the 1-D standing wave \cite{97} of the $abcd$-system (\ref{abcd}) and the 2-D traveling wave  patterns
 of  (\ref{2d})  \cite{96,98}. Here, we look for periodic solutions in $(\mathbf{x},t)$, and introduce the scaled variables $\tilde{\mathbf{x}}=\frac{ \sqrt{6}}{\lambda}\mathbf{x}$ and $\tilde{t}=\frac{ \sqrt{6}}{T}t$ with $\lambda/\sqrt{6}$ and $T/\sqrt{6}$ being the wave length and the time period. By dropping the tilde for the re-scaled Boussinesq system of (\ref{2d}), one obtains
\begin{equation}\label{nonl}
\begin{cases}
\eta_t+\beta \nabla\cdot\textbf{v}+\beta \nabla \cdot (\eta \textbf{v} ) -\omega\Delta \eta_t=0,\\
\textbf{v}_t+\beta \nabla \eta + \frac{1}{2}\beta \nabla |\textbf{v}|^2-\omega \Delta \textbf{v}_t=0,
\end{cases}
\end{equation}
where $\omega=(1/\lambda)^2$ and $\beta$ is a positive parameter defined as
\begin{equation*} \beta=T/\lambda.
\end{equation*}
We are looking for solution of system (\ref{nonl}) under the form of 2-dimensional standing waves, i.e. $\eta$ and $\textbf{v}$ are periodic functions in both time $t$ and space $\textbf{x}$, where $\textbf{x}=(x_1,x_2)\in \R$.

This  paper is organized as follows: In section 2, we set the linearized problem and study its solution near zero. In section 3, we apply the Lyapunov-Schmidt method to obtain the bifurcation standing waves with bifurcating parameter $\beta$.

\section{Linearized operator}
For fixed $\omega>0$, we start by studying the linearized system
\begin{equation}\label{linear0}
    \begin{cases}
    \eta_t+\beta \nabla\cdot\textbf{v}-\omega\Delta \eta_t=\nabla \cdot \textbf{f}\\
    \textbf{v}_t+\beta \nabla \eta -\omega \Delta \textbf{v}_t= \nabla g,
    \end{cases}
\end{equation}
and we denote it as
\begin{equation}\label{linear}
    \mathcal{L} U= \mathcal{G} F
\end{equation}
where
\begin{align*}
    &U=(\eta,\textbf{v})^T, \ \ \ \ \ F=(g,\textbf{f})^T,\\
    &\mathcal{L} U=\left(
                     \begin{array}{c}
                       \eta_t+\beta\nabla\cdot \textbf{v} -\omega\Delta \eta_t\\
                       \textbf{v}_t+\beta \nabla \eta -\omega \Delta \textbf{v}_t\\
                     \end{array}
                   \right), \ \ \ \ \mathcal{G} F=(\nabla\cdot \textbf{f}, \nabla g)^T.
\end{align*}

To simplify the problem, we look for solutions $\eta$ being odd and $\mathbf{v}$ being even in $t$, it then  follows from \eqref{linear0} that $g$ is even and $\mathbf{f}$ is odd in time.
We now write the Fourier series,
\begin{equation}\label{fs1}
    \eta=\sum_{q\geq 0, \textbf{p}\in \Gamma} \eta_{\textbf{p}q}e^{i\textbf{p}\cdot \textbf{x}}\cos{qt},
\end{equation}
\begin{equation}\label{fs2}
    \textbf{v}=\sum_{q> 0, \textbf{p}\in \Gamma} \textbf{v}_{\textbf{p}q}e^{i\textbf{p}\cdot \textbf{x}}\sin{qt},
\end{equation}
\begin{equation}\label{fs3}
    g=\sum_{q\geq0, \textbf{p}\in \Gamma} g_{\textbf{p}q}e^{i\textbf{p}\cdot \textbf{x}}\cos{qt},
\end{equation}
\begin{equation}\label{fs4}
    \textbf{f}=\sum_{q> 0, \textbf{p}\in \Gamma} \textbf{f}_{\textbf{p}q}e^{i\textbf{p}\cdot \textbf{x}}\sin{qt},
\end{equation}
where $q\in \Z$ and $\Gamma$ is a hexagonal lattice of the $x_1x_2$-plane defined by two vectors $\mathbf{P}_1$ and $\mathbf{P}_2$ such that, for $\mathbf{p}\in\Gamma$, one can write
\begin{equation}\label{lattice}
 \mathbf{p}=(p_1,p_2)=n_1\mathbf{P}_1+n_2\mathbf{P}_2, n_1, n_2\in \Z.
\end{equation}
We denote  $ \theta_1$ and $-\theta_2$  to be angle of the two vectors $\mathbf{P}_1$ and $\mathbf{P}_2$  with $x_1$  axis for $0<\theta_1, \theta_2<\pi$. Thus, we can write
\begin{equation*}
\mathbf{P}_1:=k_1(1,\tau_1), \ \ \mathbf{P}_2:=k_2(1,-\tau_2),
\end{equation*}
where $k_j>0$ and $\tau_j=\tan{\theta_j}>0$ with $j=1,2$. According to \eqref{lattice}, one has
\begin{equation}\label{m1m2}
p_1:=p_1(n_1,n_2)=k_1n_1+k_2n_2 ,\ \ p_2:= p_2(n_1,n_2)= \tau_1 k_1n_1-\tau_2k_2n_2.
\end{equation}
In addition,  it is no harm to assume that $\mathbf{v}_{\mathbf{p}0}=\mathbf{f}_{\mathbf{p}0}=\mathbf{0}:=(0,0)$ for $\mathbf{p}\in \Gamma$ if they are needed in the paper.

We now substitute  \eqref{fs1}-\eqref{fs4} into \eqref{linear0} to solve the linear system, and obtain:
  \begin{itemize}
   \item for $q>0$ and $\mathbf{p}\in \Gamma$
\begin{equation}\label{force}
\begin{cases}
-q(1+\omega|\textbf{p}|^2)\eta_{\textbf{p}q}+i\beta \textbf{p} \cdot \textbf{v}_{\textbf{p}q}= i\textbf{p}\cdot \textbf{f}_{\textbf{p}q},\\
i\beta \textbf{p} \eta_{\textbf{p}q}+ q (1+\omega|\textbf{p}|^2)\textbf{v}_{\textbf{p}q}= i\textbf{p} g_{\textbf{p}q};
\end{cases}
\end{equation}
   \item for $q=0$ and $\mathbf{p}\neq (0,0)$, $\eta_{\mathbf{p}q}=\frac{1}{\beta}g_{\mathbf{p}q}$ and $\mathbf{v}_{\mathbf{p}q}=(0,0)$;
   \item for $q=0$ and $\mathbf{p}= (0,0)$,  $\eta_{\mathbf{p}q}$ is arbitrary and $\mathbf{v}_{\mathbf{p}q}=(0,0)$.
 \end{itemize}
 Set
\begin{equation*}
    \Delta(\textbf{p},q)=q^2(1+\omega|\textbf{p}|^2)^2-\beta^2|\textbf{p}|^2.
\end{equation*}
For $\Delta(\textbf{p},q)\neq 0$, according to \eqref{force}, we have
\begin{equation}\label{s1}
    \eta_{\textbf{p}q}=-\Delta^{-1}[ iq(1+\omega|\textbf{p}|^2)(\textbf{p}\cdot \mathbf{f}_{\textbf{p}q})+ \beta |\textbf{p}|^2 g_{\textbf{p}q}],
\end{equation}
\begin{equation}\label{s2}
\textbf{v}_{\textbf{p}q}=-\Delta^{-1}[-iq(1+\omega|\textbf{p}|^2)\textbf{p}g_{\textbf{p}q}+\beta \textbf{p}(\textbf{p}\cdot \mathbf{f}_{\textbf{p}q})].
\end{equation}
For $\Delta(\textbf{p},q)=0$, according to the second equation of (\ref{force}) we have
\begin{equation*}
\textbf{v}_{\textbf{p}q}=\frac{-i\beta \textbf{p} \eta_{\textbf{p}q}+i\textbf{p} g_{\textbf{p}q}}{ q (1+\omega|\textbf{p}|^2)},
\end{equation*}
then it follows from the first equation of (\ref{force}) that
\begin{equation*}
-\eta_{\textbf{p}q}\left(q^2(1+\omega|\textbf{p}|^2)^2-\beta^2|\textbf{p}|^2\right)=iq(1+\omega|\textbf{p}|^2)\textbf{p}\cdot \textbf{f}_{\textbf{p}q}+\beta |\textbf{p}|^2g_{\textbf{p}q}.
\end{equation*}
Hence, based on $q^2(1+\omega|\textbf{p}|^2)^2-\beta^2|\textbf{p}|^2=0$, i.e.  $q(1+\omega|\textbf{p}|^2)-\beta|\textbf{p}|=0$, we have the compatibility condition for the system (\ref{force}) in the form
\begin{equation*}
iq(1+\omega|\textbf{p}|^2)\textbf{p}\cdot \textbf{f}_{\textbf{p}q}+\beta |\textbf{p}|^2g_{\textbf{p}q}=0,
\end{equation*}
or
\begin{equation}\label{comp1}
   |\mathbf{p}|g_{\textbf{p}q}+i \textbf{p}\cdot \textbf{f}_{\textbf{p}q}=0.
\end{equation}
%
%
%
If such condition is satisfied, one has the solution of (\ref{force}) as,
\begin{align}
    &\eta_{\textbf{p}q}=c|\mathbf{p}|+\frac{1}{\beta}g_{\textbf{p}q}=c\sqrt{p_1^2+ p_2^2}+\frac{1}{\beta}g_{\textbf{p}q},\label{1}\\
    &\textbf{v}_{\textbf{p}q}=-i c\mathbf{p} =-ic(p_1,p_2)\label{2}
\end{align}
where $c$ is arbitrary in $\mathbb{C}$.

For $q>0$, the linearized operator $\mathcal L$ defined in (\ref{linear}) has a  non-trivial kernel if there exists a pair of $(\textbf{p},q)$ satisfying $\Delta(\textbf{p},q)=0$.
 The following lemmas provide the detail information on  $(\mathbf{p}, q)$ when a kernel exists.
\begin{lemma}\label{U}
Given $\beta>0$, let
  $$p_1:=k_1n_1+k_2n_2, \quad p_2:=\tau_1k_1n_1-\tau_2k_2n_2,$$
 then the  set
\begin{equation*}
    \Omega_{\beta}:=\{(n_1,n_2,q)\in \Z^2 \times \Z^+ ,q\left(1+\omega (p_1^2+p_2^2)\right)-\beta  \sqrt{p_1^2+p_2^2}=0\}
\end{equation*}
 is either empty or finite.
\end{lemma}

\begin{proof}
For  $\beta$   positive, $(n_1,n_2,q)\in \Omega_{\beta} $ implies
\begin{equation*}
    q=\frac{\beta \sqrt{p_1^2+p_2^2}}{1+\omega (p_1^2+p_2^2)}\leq \frac{\beta}{\frac{1}{\sqrt{p_1^2+p_2^2}}+\omega \sqrt{p_1^2+p_2^2}}\leq \frac{\beta}{2\sqrt{\omega}}.
\end{equation*}
Hence, the only possible values are $q= 1,2,...,[\frac{\beta}{2\sqrt{\omega}}]$. For each  $q$ such that $(n_1,n_2,q)\in \Omega_{\beta} $, it implies
\begin{equation}\label{O}
    \omega  q  \left(\sqrt{p_1^2+p_2^2}\right)^2 -\beta \left(\sqrt{p_1^2+p_2^2}\right)+q=0.
\end{equation}
Thus, there are at most two possible values of $\sqrt{p_1^2+p_2^2}$.

Now, for given $K>0$, we consider the number of pairs  for $(n_1,n_2)\in \Z^2$ satisfying:
\begin{equation}\label{K}
    K=p_1^2+p_2^2=(k_1n_1+k_2n_2)^2+(\tau_1k_1n_1-\tau_2k_2n_2)^2.
\end{equation}

\begin{itemize}
  \item If $n_1\cdot n_2\geq 0$, according to \eqref{K} and $(n_1,n_2)\in \Z^2$, one has
\[(k_1|n_1|+k_2|n_2|)^2\leq K,\]
which follows that the number of pairs of $(n_1, n_2)$ is finite.
 \item If $n_1\cdot n_2<0$, according to \eqref{K} and $(n_1,n_2)\in \Z^2$, one has
\[(\tau_1k_1|n_1|+\tau_2k_2|n_2|)^2\leq K,\]
which follows that the number of pairs of $(n_1, n_2)$ is finite.
\end{itemize}
 Therefore, for each value of $\sqrt{p_1^2+p_2^2}$ (which is finite), we have finite choices of $(n_1,n_2)\in \Z^2$. The proof is now complete.
\end{proof}



We now study the norm of the operator $\mathcal{L}$.

\begin{lemma}\label{M}
Given $\beta>0$,
there exists a constant $M>0$, depending only on $\beta$, such that for any $(n_1,n_2,q)\in \Z^2\times \Z^+$, $(n_1,n_2,q)\notin \Omega_{\beta}$ with $(p_1,p_2)$ defined in (\ref{m1m2}),  we have
\begin{equation*}
    \frac{\sqrt{p_1^2+p_2^2}}{\left|q(1+\omega  (p_1^2+p_2^2))-\beta \sqrt{p_1^2+p_2^2}\right|}\leq M.
\end{equation*}
\end{lemma}
\begin{proof}
We first consider the pairs $(n_1,n_2,q)\in \Z^2\times \Z^+$ with $(n_1,n_2,q)\notin \Omega_{\beta}$   satisfying $q \sqrt{p_1^2+p_2^2}> 2 \beta/\omega$, then
\begin{equation*}
     q(1+\omega (p_1^2+p_2^2))-\beta \sqrt{p_1^2+p_2^2}\geq q +\beta \sqrt{p_1^2+p_2^2}>0,
\end{equation*}
therefore,
\begin{align*}
     \frac{\sqrt{p_1^2+p_2^2}}{\left|q(1+\omega  (p_1^2+p_2^2))-\beta \sqrt{p_1^2+p_2^2}\right|}\leq   \frac{\sqrt{p_1^2+p_2^2}}{q +\beta \sqrt{p_1^2+p_2^2}}\leq \frac{1}{\beta}.
\end{align*}
Next, we consider the case for $(n_1,n_2,q)\in  \Z^2\times \Z^+$, $(n_1,n_2,q)\notin \Omega_{\beta}$ and $(n_1,n_2)\neq (0,0)$ satisfies
\begin{equation}\label{alp}
    q \sqrt{p_1^2+p_2^2}\leq 2 \beta/\omega.
\end{equation}
We claim that, in such a case, there are finite pairs of $(n_1,n_2,q)$.
\begin{description}
  \item[(A)] We notice that
  $\sqrt{p_1^2+p_2^2}=0$
  implies that $(n_1,n_2)=(0,0)$. According to  (\ref{m1m2}),
  \[p_1=k_1n_1+k_2n_2=0,\quad p_2=\tau_1k_1n_1-\tau_2k_2n_2=0\]
  and
\[\begin{bmatrix}
   k_1& k_2 \\
    \tau_1k_1& \tau_2k_2\end{bmatrix}=-(\tau_1+\tau_2)k_1k_2\neq 0,\]
     it follows that the only solution is the trivial solution. Therefore, $\sqrt{p_1^2+p_2^2}>0$ for any $(n_1,n_2)\neq (0,0) $.
  \item[(B)] There are finite $(n_1,n_2)$ satisfying \eqref{alp} for any $q\in\Z^+$. This can be seen by noticing $q\geq 1$ and
       $$\sqrt{p_1^2+p_2^2}\leq \frac{\beta}{2\omega}\cdot \frac1q\leq \frac{\beta}{2\omega},$$
       along with the process in \eqref{K}.
\end{description}
The proof of the claim now follows from combining
   \[ q\leq \frac{\beta}{2\omega}\cdot \frac{1}{\sqrt{p_1^2+p_2^2}}, \quad q\in  \Z^+\]
  since $\sqrt{p_1^2+p_2^2}>0$ for  $(n_1,n_2)\neq (0,0) $ (see (A)) and the conclusion in (B).
Because the number of pairs of $(n_1,n_2,q)$ is finite and the denominator
\begin{equation*}
    \left|q(1+\omega  (p_1^2+p_2^2))-\beta \sqrt{p_1^2+p_2^2}\right|\neq 0,
\end{equation*}
due to $(n_1,n_2,q)\notin \Omega_{\beta}$, one can obtain the upper bound, $M$, depending on $\beta$.

Finally, for $(n_1,n_2,q)\in  \Z^2\times \Z^+$, $(n_1,n_2)= (0,0)$ and $(n_1,n_2,q)\notin \Omega_{\beta}$, we can check the bound is $0$. The proof is now complete.
\end{proof}


Now, recall that
$$p_1=k_1n_1+k_2n_2, \quad p_2=\tau_1k_1n_1-\tau_2k_2n_2, \mbox{ with } n_1, n_2\in \Z $$
and set
$$ {\mathcal H}^s:=H^s\{(\R/(2\pi/p_1))\times (\R/(2\pi/p_2))\times(\R/(2\pi)\Z)\},$$
we introduce the Sobolev spaces
\begin{equation*}
    H^s_{\natural,e}=\left\{u=\sum_{\textbf{n} \in \Z^2, q\geq0} u_{\textbf{n}q} e^{ip_1x_1}e^{ip_2x_2}\cos{qt}\in {\mathcal H}^s, \mbox{ where } q\in \Z.\right\},
\end{equation*}
\begin{equation*}
    H^s_{\natural,o}=\left\{u=\sum_{\textbf{n} \in \Z^2, q>0} u_{\textbf{n}q} e^{ip_1x_1}e^{ip_2x_2}\sin{qt}\in {\mathcal H}^s, \mbox{ where }  q\in \Z.\right\}.
\end{equation*}
We  define an operator $\mathcal{T}$:
 \[U=\mathcal{T} F\]
 acting in $Q_s:=H^s_{\natural,e}\times H^s_{\natural,o} \times H^s_{\natural,o}$, for any $s\geq 0$, solving
 \[\mathcal{L} U=\mathcal{G}F,\]
 provided the compatibility condition
 \begin{equation}\label{compatible11}
 |\mathbf{p}|g_{\textbf{n}q}+i \textbf{p}\cdot \textbf{f}_{\textbf{n}q}=0,\quad \mbox{for}\ \ \mathbf{p}=(p_1,p_2), \ \ \mathbf{n}=(n_1,n_2)
 \end{equation}
is satisfied, such that,

\begin{itemize}
  \item for $(n_1,n_2,q)$ satisfying  $q\neq 0$ and $\Delta(n_1,n_2,q)\neq 0$,
\begin{align}\label{t1}
&\mathcal{T} F_{\textbf{n}q}= U_{\textbf{n}q}=(\eta_{\textbf{n}q},\textbf{v}_{\textbf{n}q}),
\end{align}
with $\eta_{\textbf{n}q}$ and $\textbf{v}_{\textbf{n}q}$ given by (\ref{s1}) and (\ref{s2}), that is,
\begin{align}
  \eta_{\textbf{n}q}=&-\Delta^{-1}(p_1,p_2)\cdot\big[ iq(1+\omega (p_1^2+p_2^2)) \mathbf{f}_{\textbf{n}q}+ \beta (p_1,p_2) g_{\textbf{n}q}\big],\label{s11}\\
\textbf{v}_{\textbf{n}q}=&-\Delta^{-1}(p_1,p_2)\big[-iq(1+\omega (p_1^2+p_2^2))g_{\textbf{n}q}
+\beta (p_1,p_2)\cdot \mathbf{f}_{\textbf{n}q}\big];\label{s22}
\end{align}
  \item for $(n_1,n_2,q)$ satisfying $\Delta(n_1,n_2,q)=0$ or $q=0$ but $(n_1,n_2,q)\neq(0,0,0)$,
\begin{align}\label{t2}
 \mathcal{T}F_{\textbf{n}q}= (\frac{1}{\beta}g_{\textbf{n}q},0,0);
\end{align}
  \item for $(n_1,n_2,q)=(0,0,0)$
\begin{align}\label{t3}
 \mathcal{T}F_{\textbf{n}q}= (0,0,0).
\end{align}
\end{itemize}
Notice that the operator $\mathcal{T}$ is defined here, even for $F=(g,\mathbf{f})$ not satisfying the compatibility condition.
\begin{lemma}
Given positive numbers  $\beta$, if $\Omega_{\beta}$ has finite elements and $\mathcal{T}$ is defined as in (\ref{t1})-(\ref{t3}),  then the  operator $\mathcal{T}$ is bounded from  $H^s_{\natural,e}\times H^s_{\natural,o} \times H^s_{\natural,o}$ to  $H^s_{\natural,e}\times H^s_{\natural,o} \times H^s_{\natural,o}$ for any $s\geq 0$.
\end{lemma}
\begin{proof}
According to (\ref{s11}) and (\ref{s22}) for $\Delta(\mathbf{n},q)\neq 0$ and $q\neq 0$, we have
\begin{align*}
|\eta_{\textbf{n}q}|+|\textbf{v}_{\textbf{n}q}|\leq&  \frac{C\left(\left|q(p_1,p_2)\right|(1+\omega (p_1^2+p_2^2))+\beta (p_1^2+p_2^2)\right)}{q^2\left(1+\omega (p_1^2+p_2^2)\right)^2-\beta^2 (p_1^2+p_2^2)}(|g_{\textbf{n}q}|+|\textbf{f}_{\textbf{n}q}|)\\
\leq&\frac{C\sqrt{p_1^2+p_2^2}}{\left|q(1+\omega  (p_1^2+p_2^2))-\beta \sqrt{p_1^2+p_2^2}\right|}(|g_{\textbf{n}q}|+|\textbf{f}_{\textbf{n}q}|).
\end{align*}
Then, it follows from Lemma \ref{M}
\begin{equation*}
|\eta_{\textbf{n}q}|+|\textbf{v}_{\textbf{n}q}|\leq C M(|g_{\textbf{n}q}|+|\textbf{f}_{\textbf{n}q}|).
\end{equation*}
The proof of case in \eqref{t2} and \eqref{t3} can also be readily checked.
\end{proof}

We now move on to consider the kernel of $\mathcal{L}_0:=\mathcal{L}_{\beta_0}$, assume that $\beta_0$ is given such that $\Omega_{\beta_0}$ has finite elements. We  denote the finite pairs of $(n_1,n_2,q)\in \Omega_{\beta_0}$  as $(n_1^{(j)},n_2^{(j)},q^{(l)})$, for $n_1^{(j)}$, $n_2^{(j)} \in \Z$, $q^{(l)}\in \Z^+$ and $j=1,2,...,N_1$, $l=1,2,...N_2$. In addition,  we shall notice that if $(n_1,n_2,q) \in \Omega_{\beta_0}$, then $(-n_1,-n_2,q)\in \Omega_{\beta_0}$. Let
\begin{align*}
    &\xi_{j,l}:=\xi_{(n_1^{(j)},n_2^{(j)},q^{(l)})},\ \ \ \ \overline{\xi_{j,l}}:=\xi_{(-n_1^{(j)},-n_2^{(j)},q^{(l)})},
        \end{align*}
where
\begin{align*}
\xi_{(n_1^{(j)},n_2^{(j)},q^{(l)})}=\bigg(\sqrt{(p^{(j)}_1)^2+(p^{(j)}_2)^2}\cos{q^{(l)}t},-i&p_1^{(j)}\sin{q^{(l)}t},\\
&-i  p_2^{(j)}\sin{q^{(l)}t}\bigg)e^{ip_1^{(j)}x_1}e^{ip_2^{(j)} x_2}
\end{align*}
with $p^{(j)}_1=k_1n_1^{(j)}+k_2n_2^{(j)}$ and $p^{(j)}_2=\tau_1 k_1n_1^{(j)}+\tau_2 k_2n_2^{(j)}$
for $j=1,2,...,N_1$ and $l=1,2,...,N_2$.
 It then follows from (\ref{1}) and (\ref{2}) that
\begin{equation*}
    \mathcal{L}_0  \xi_{j,l}=0, \quad \mathcal{L}_0  \overline{\xi_{j,l}}=0.
\end{equation*}
Therefore, the kernel of ${\mathcal L}_0$ is spanned by $$\{\zeta_0 ,\xi_{j,l}, \overline{\xi_{j,l}}, j=1,2,...,N_1, l=1,2,...N_2\}$$
where $\zeta_0:= (1,0,0)$ is the ``trivial" solution of \eqref{linear0}.

Next, we rewrite the compatibility conditions \eqref{compatible11}. According to the scalar product in $H^0_{\natural,e}\times H^0_{\natural,o} \times H^0_{\natural,o}$, the compatibility condition of $\mathcal{L}_0$ can read as
\begin{equation}\label{compatible1}
 \langle F,\xi_{(\textbf{n},q)}\rangle=0,\ \ \ \ \mbox{for } \Delta(\textbf{n},q)=0,
\end{equation}
 where  $F\in Q_s$ and
\begin{equation*}
    F=\sum_{(\textbf{n},q)\in \Z^2\times \N_0} F_{\textbf{n}q} e^{ip_1x_1}e^{ip_2x_2},
\end{equation*}
with
\begin{equation*}
    F_{\textbf{n}q}=(g_{\textbf{n}q}\cos{qt},(\textbf{f}_{\textbf{n}q})_1\sin{qt},(\textbf{f}_{\textbf{n}q})_2\sin{qt}),
\end{equation*}
for $\N_0=\N\cup\{0\}$.  Hence, for $\xi_{j,l}\in Q_s$ with any $s$. It follows that the compatibility conditions (\ref{compatible1}) can now be reduced to
\begin{equation*}
    \langle F,\xi_{j,l}\rangle=0, \ \ \ \ \  \langle F,\overline{\xi_{j,l}}\rangle=0, \ \ \ \ \mbox{   for } j=1,2,...N_1,  \ \ l=1,2...N_2.
\end{equation*}
Finally, according to the previous settings, we have the following theorem.

\begin{theorem}\label{operator}
Assume that $\beta_0$ is such that $\Omega_{\beta_0}$ have finite elements \[(n_1^{(j)},n_2^{(j)},q^{(l)})\]
in $\Z^2\times \Z^+$ with $j=1,2,...,N_1$ and $l=1,2..,N_2$. Then, for any given
\begin{equation*}
    F=(g,\mathbf{f})\in Q_s:=H^{s}_{\natural,e}\times H^{s}_{\natural,o} \times H^{s}_{\natural,o},\ \ s\geq 0,
\end{equation*}
such that it satisfies the following compatibility conditions
\begin{equation*}
    \langle F,\xi_{j,l}\rangle=0, \ \ \ \ \  \langle F,\overline{\xi_{j,l}}\rangle=0, \ \ \ \ \mbox{   for } \ \ j=1,2,...,N_1,\ \ l=1,2,...,N_2,
\end{equation*}
the general solution $U=(\eta,\textbf{v})\in Q_s$ of the system
\begin{equation*}
    L_0 U= \mathcal{G} F
\end{equation*}
is given by
\begin{equation*}
    U=\mathcal{T}F +\sum^{N_1}_{j=1}\sum^{N_2}_{l=1} \left(A_{j,l}\xi_{j,l}+\overline{A_{j,l}\xi_{j,l}}\right)+C \zeta_0
\end{equation*}
where
\begin{align*}
    &\zeta_0=(1,0,0),\\
    &\xi_{j,l}=\left(\sqrt{(p^{(j)}_1)^2+(p^{(j)}_2)^2}\cos{q^{(l)}t},-ip_1^{(j)}\sin{q^{(l)}t},-i p_2^{(j)}\sin{q^{(l)}t}\right)e^{ip_1^{(j)}x_1}e^{ip_2^{(j)} x_2},
\end{align*}
with  $p^{(j)}_1=k_1n_1^{(j)}+k_2n_2^{(j)}$ and $p^{(j)}_2=\tau_1 k_1n_1^{(j)}+\tau_2 k_2n_2^{(j)}$, $A_{j,l}\in \mathbb{C}$, $C\in \R$, $j=1,2,...,N_1 $ and $l=1,2,...,N_2$.
\end{theorem}
\begin{remark}
For simplicity, in the next section we only consider the case for $\beta_0$ such that we only have  $(0, \pm 1, q)$ and $(\pm 1, 0, q)$ in $\Omega_{\beta_0}$. However, the same strategy can be carried out for more kernels.
\end{remark}

\section{Bifurcation problem}

The aim of this section is to use the Lyapunov-Schmidt method to obtain  bifurcation standing waves. Assume that, for given $\beta_0>0$, $(\pm1,0,q)$ and $(0,\pm1,q)$ are the only elements in $\Omega_{\beta_0}$. We denote the relative kernel for ${\mathcal L}_0$ as $\zeta_0$, $\xi_1$, $\overline{\xi}_1$, $\xi_2$ and $\overline{\xi}_2$.  Let us consider (\ref{nonl}) for parameters $ \beta= \beta_0+\mu$ with $\mu$ close to $0$. We look for nontrivial triple periodic solution $(\eta,\mathbf{v})$ in $H^{s}_{\natural,e}\times H^{s}_{\natural,o} \times H^{s}_{\natural,o}$. One shall notice that for $k\geq 2$ (Sobolev imbedding theorem)
\begin{equation*}
    (\frac12|\mathbf{v}|^2,\eta \mathbf{v})\in H^{s}_{\natural,e}\times H^{s}_{\natural,o} \times H^{s}_{\natural,o},
\end{equation*}
hence, we write the full system as
\[{\mathcal L}_0 U={\mathcal G}F\]
where  $F= \left( \begin{array}{c}
                       g \\
                       \textbf{f}\\
                     \end{array}
                   \right)$ and
\begin{align*}
&g=-\mu \eta-\frac12(\beta_0+\mu)|\mathbf{v}|^2,\\
&\mathbf{f}=-\mu \textbf{v}-(\beta_0+\mu)\eta \mathbf{v}.
\end{align*}
Then $F$ has the properties required by Theorem \ref{operator} once the compatibility conditions are satisfied.

Let $U=\left( \begin{array}{c}
                       \eta \\
                       \textbf{v}\\
                     \end{array}
                   \right)\in Q_s$, we can write the system as
\begin{equation}\label{sys1}
    \mathcal{L}_0 U =\mathcal{G}\left[-\mu U +(\beta_0+\mu) B(U,U)\right],
\end{equation}
where
\begin{equation*}
    B(U,U)=(\frac12|\mathbf{v}|^2,\eta \mathbf{v})^T.
\end{equation*}
Recall that the kernel of ${\mathcal L}_0=span\{\zeta_0, \xi_1, \bar{\xi}_1,\xi_2, \bar{\xi}_2\}$ where
\begin{align*}
\zeta_0&=(1,0,0),\\
\xi_1&=\left(k_1\sqrt{1+\tau_1^2}\cos{qt},-ik_1\sin{qt},-i\tau_1k_1\sin{qt}\right)e^{ik_1x_1}e^{i\tau_1 k_1 x_2},\\
\xi_2&=\left(k_2\sqrt{1+\tau_2^2}\cos{qt},-ik_2\sin{qt},i\tau_2k_2\sin{qt}\right)e^{ik_2x_1}e^{-i\tau_2 k_2 x_2}.
\end{align*}
We now decompose $U\in Q_s$ as
\begin{equation}\label{u0}
    U=X+V:=A_1 \xi_1+\bar{A}_1\bar{\xi}_1+A_2\xi_2+\bar{A}_2\bar{\xi}_2+C \zeta_0+V.
\end{equation}
The unknown $A_1$, $A_2$, $V$ are functions of parameter  $\mu$. In addition,  $V$ has the same properties as $\mathcal{T}F$ defined in \eqref{t2}-\eqref{t3}, more precisely,
\begin{align*}
V_{\mathbf{n}q}&=(0,0,0), \quad \mbox{for}\quad  (\mathbf{n}, q)=(0,0,0);\\
V_{\mathbf{n}q}&=(\eta_{\mathbf{n}q},0,0), \quad \mbox{for} \quad  \Delta(\mathbf{n},q)=0 \quad \mbox{or} \quad q=0 \quad \mbox{but}  \quad (\mathbf{n}, q)\neq (0,0,0).
\end{align*}
 According to Theorem \ref{operator}, if  the compatibility conditions
\begin{equation}\label{e20}
    \langle\mu U+(\beta_0+\mu)B(U,U),\xi_1\rangle=0
\end{equation}
and
\begin{equation}\label{e30}
    \langle\mu U +(\beta_0+\mu)B(U,U),\xi_2\rangle=0
\end{equation}
are satisfied, we can substitute \eqref{u0} into the nonlinear system \eqref{sys1} and obtain
\begin{equation}\label{e1}
    V+\mu  \mathcal{T} U+(\beta_0+\mu) \mathcal{T}B(U,U)=0,
\end{equation}
Here, the Lyapumov-Schmidt reduction is applied and \eqref{e1} is the projected equation.  We seek for a ``solution'' from the ``projection'' equation \eqref{e1} and
 the compatibility conditions will lead to relations between
  $A_1$, $A_2$, $C$ and $\mu$.  For more details of the Lyapunov-Schmidt method, the reader may refer to \cite{iooss}.

Now, It follows from \eqref{e1} and the implicit function theorem for $A_1$, $A_2$, $C$ and $\mu$ close enough to 0 that
\begin{equation*}
    V=V(A_1, A_2, C ,\mu) \in  Q_s,
\end{equation*}
where
\begin{align}
    V=-\mu\mathcal{T}X
    -\beta_0\mathcal{T}B(X,X)+O\left(|\mu|^2\|X\|+|\mu|\|X\|^2+\|X\|^3\right)\nonumber.
\end{align}
Let us denote
\begin{equation}\label{u}
U=X+V:=X-V_1-V_2+O\left(|\mu|^2\|X\|+|\mu|\|X\|^2+\|X\|^3\right),
\end{equation}
where
\begin{align} \label{V12}
V_1=\mu\mathcal{T}X, \quad
V_2=\beta_0\mathcal{T}B(X,X):=V_3+V_4.
\end{align}
Then, with simple but tedious computations, we obtain
\begin{align*}
    V_1=\frac{1}{\beta_0}\mu   \Big(&A_1 k_1 \sqrt{1+\tau_1^2} e^{ik_1x_1}e^{ik_1\tau_1 x_2}+\bar{A}_1 k_1 \sqrt{1+\tau_1^2}e^{-ik_1x_1}e^{-ik_1\tau_1 x_2}\\
    &+A_2 k_2 \sqrt{1+\tau_2^2}e^{ik_2x_1}e^{-ik_2\tau_2 x_2}+\bar{A}_2k_3 \sqrt{1+\tau_2^2}e^{-ik_2x_1}e^{ik_2\tau_2 x_2}\Big)\zeta_0\cos{qt},
 \end{align*}
 \begin{align*}
  V_3=&-\frac{1}{4}\Big[k_1^2(1+\tau_1^2)\big(A_1^2e^{2ik_1x_1}e^{2ik_1\tau_1 x_2}+\bar{A}_1^2  e^{-2ik_1x_1}e^{-2ik_1\tau_1 x_2}\big)+k_2^2(1+\tau_2^2)\big(A_2^2\\
&\cdot e^{2ik_2x_1}e^{-2ik_2\tau_2 x_2}+\bar{A}_2^2e^{-2ik_2x_1}e^{2ik_2\tau_2 x_2}\big)+2k_1k_2(1-\tau_1\tau_2)\big(A_1A_2 e^{i(k_1+k_2) x_1}\\
&\cdot e^{i(k_1\tau_1-k_2\tau_2) x_2}+\bar{A}_1\bar{A}_2e^{-i(k_1+k_2) x_1}e^{-i(k_1\tau_1-k_2\tau_2) x_2}-A_1\bar{A}_2e^{i(k_1-k_2) x_1}e^{i(k_1\tau_1+k_2\tau_2) x_2}\\
    &-\bar{A}_1A_2e^{-i(k_1-k_2) x_1} e^{-i(k_1\tau_1+k_2\tau_2) x_2}\big)\Big]\zeta_0,
    \end{align*}
\begin{align*}
V_4=&\beta_0\Big[A_1^2e^{2ik_1x_1}e^{2ik_1\tau_1 x_2}\delta_1+\bar{A}_1^2e^{-2ik_1x_1}e^{-2ik_1\tau_1 x_2}\bar{\delta}_1+A_2^2e^{2ik_2x_1}e^{-2ik_2\tau_2 x_2}\delta_2\\
&+\bar{A}_2^2e^{-2ik_2x_1}e^{2ik_2\tau_2 x_2}\bar{\delta}_2+A_1A_2e^{i(k_1+k_2)x_1}e^{i(k_1\tau_1-k_2\tau_2) x_2}\delta_3+\bar{A}_1\bar{A}_2e^{-i(k_1+k_2)x_1}\\
&\cdot e^{-i(k_1\tau_1-k_2\tau_2) x_2}\bar{\delta}_3+A_1\bar{A}_2e^{i(k_1-k_2)x_1}e^{i(k_1\tau_1+k_2\tau_2) x_2}\delta_4+\bar{A}_1A_2e^{-i(k_1-k_2)x_1}\\
&\cdot e^{-i(k_1\tau_1+k_2\tau_2) x_2}\bar{\delta}_4\Big],
\end{align*}
where
\begin{equation*}
    \delta_1=(a_1\cos{2qt},ia_2\sin{2qt},i\tau_1 a_2\sin{2qt}),
    \delta_2=(a_3\cos{2qt},ia_4\sin{2qt},-i\tau_2a_4\sin{2qt}),
\end{equation*}
\begin{equation*}
    \delta_3=(b_1\cos{2qt},ib_2\sin{2qt},ib_3\sin{2qt}),\ \ \ \delta_4=(b_4\cos{2qt},ib_5\sin{2qt},ib_6\sin{2qt}),
\end{equation*}
\begin{equation*}
        a_1=-\frac{\beta_0 k_1^4(1+\tau_1^2)^2+2qk_1^3(1+\tau_1^2)^{\frac32}(1+4\omega k_1^2(1+\tau_1^2))}{4q^2(1+4\omega k_1^2 (1+\tau_1^2))^2-4\beta_0^2k_1^2(1+\tau_1^2)},
\end{equation*}
\begin{equation*}
    a_2=\frac{2\beta_0 k_1^4(1+\tau_1^2)^{\frac32}+qk_1^3(1+\tau_1^2)(1+4\omega k_1^2(1+\tau_1^2))}{4q^2(1+4\omega k_1^2 (1+\tau_1^2))^2-4\beta_0^2k_1^2(1+\tau_1^2)},
\end{equation*}
\begin{equation*}
    a_3=-\frac{\beta_0 k_2^4(1+\tau_2^2)^2+2qk_2^3(1+\tau_2^2)^{\frac32}(1+4\omega k_2^2(1+\tau_2^2))}{4q^2(1+4\omega k_2^2 (1+\tau_2^2))^2-4\beta_0^2k_2^2(1+\tau_2^2)},
\end{equation*}
\begin{equation*}
    a_4=\frac{2\beta_0 k_2^4(1+\tau_2^2)^{\frac32}+qk_2^3(1+\tau_2^2)(1+4\omega k_2^2(1+\tau_2^2))}{4q^2(1+4\omega k_2^2 (1+\tau_2^2))^2-4\beta_0^2k_2^2(1+\tau_2^2)},
\end{equation*}
\begin{align*}
        b_1&=\frac{1}{2d_1}\bigg\{ 2q\big[1+\omega \left((k_1+k_2)^2+(k_1\tau_1-k_2\tau_2\right)^2)\big]\Big[-k_1k_2(k_1+k_2)(\sqrt{1+\tau_1^2}\\
        &+\sqrt{1+\tau_2^2})+k_1k_2(k_1\tau_1-k_2\tau_2)
        (\tau_2\sqrt{1+\tau_1^2}-\tau_1\sqrt{1+\tau_2^2})\Big]-\beta_0 k_1k_2(1\\
        &-\tau_1\tau_2)\left[(k_1+k_2)^2+(k_1\tau_1-k_2\tau_2)^2\right]\bigg\},
\end{align*}
\begin{align*}
        b_2&=\frac{k_1+k_2}{2d_1}\bigg\{ 2qk_1k_2(1-\tau_1\tau_2)\big[1+\omega \left((k_1+k_2)^2+(k_1\tau_1-k_2\tau_2\right)^2)\big]\\
        &+\beta_0\Big[k_1k_2(k_1+k_2)\big(\sqrt{1+\tau_1^2}+\sqrt{1+\tau_2^2}\big)-k_1k_2
        (k_1\tau_1-k_2\tau_2)\\
        &\cdot \big(\tau_2\sqrt{1+\tau_1^2}-\tau_1\sqrt{1+\tau_2^2}\big)\Big]\bigg\},
\end{align*}
\begin{equation*}
    b_3= \frac{k_1\tau_1-k_2\tau_2}{k_1+k_2}b_2,
\end{equation*}
\begin{align*}
        b_4&=\frac{1}{2d_2}\bigg\{ 2q\big[1+\omega \left((k_1-k_2)^2+(k_1\tau_1+k_2\tau_2\right)^2)\big]\Big[k_1k_2(k_1-k_2)(\sqrt{1+\tau_1^2}\\
        &-\sqrt{1+\tau_2^2})-k_1k_2(k_1\tau_1+k_2\tau_2)
        (\tau_2\sqrt{1+\tau_1^2}+\tau_1\sqrt{1+\tau_2^2})\Big]+\beta_0 k_1k_2(1\\
        &-\tau_1\tau_2)\left[(k_1-k_2)^2+(k_1\tau_1+k_2\tau_2)^2\right]\bigg\},
\end{align*}
\begin{align*}
        b_5&=\frac{k_1-k_2}{2d_2}\bigg\{ 2qk_1k_2(\tau_1\tau_2-1)\big[1+\omega \left((k_1-k_2)^2+(k_1\tau_1+k_2\tau_2\right)^2)\big]\\
        &-\beta_0\Big[k_1k_2(k_1-k_2)\big(\sqrt{1+\tau_1^2}-\sqrt{1+\tau_2^2}\big)-k_1k_2
        (k_1\tau_1+k_2\tau_2)\\
        &\cdot \big(\tau_2\sqrt{1+\tau_1^2}+\tau_1\sqrt{1+\tau_2^2}\big)\Big]\bigg\},
\end{align*}
\begin{equation*}
    b_6= \frac{k_1\tau_1+k_2\tau_2}{k_1-k_2}b_5,
\end{equation*}
\begin{equation*}
    d_1=4q^2 \big[1+\omega \left((k_1+k_2)^2+(k_1\tau_1-k_2\tau_2\right)^2)\big]^2-\beta_0^2\left((k_1+k_2)^2
    +(k_1\tau_1-k_2\tau_2)\right),
\end{equation*}
\begin{equation*}
    d_2=4q^2 \big[1+\omega \left((k_1-k_2)^2+(k_1\tau_1+k_2\tau_2\right)^2)\big]^2-\beta_0^2\left((k_1-k_2)^2
    +(k_1\tau_1+k_2\tau_2)\right).
\end{equation*}
Now, we substitute \eqref{u} into \eqref{e20} and  obtain
\begin{align*}
&\langle\mu U+(\beta_0+\mu)B(U,U),\xi_1\rangle\\
=&\mu\langle  X-V_1-V_2,\xi_1\rangle+(\beta_0+\mu)\langle B(X+V,X+V),\xi_1  \rangle=0,
\end{align*}
with
\begin{align*}
  B(X+V,X+V)=& B(X-V_1-V_2,X-V_1-V_2)\\
=&B(X,X)+B(V_1,V_1)+B(V_2,V_2)-2B(X,V_1)-2B(X,V_2)\\
&+2B(V_1,V_2),
\end{align*}
where   the terms of $O\left(|\mu|^2\|X\|+|\mu|\|X\|^2+\|X\|^3\right)$  can be neglected. Observe that, $2B(U_1,U_2)=(\mathbf{v}_1\cdot\mathbf{v}_2,\eta_1\mathbf{v}_2+\eta_2\mathbf{v}_1)^T $,
\begin{equation*}
\mu\langle X,\xi_1\rangle=\mu\langle A_1 \xi_1,\xi_1\rangle=A_1\mu k_1^2(1+\tau_1^2),
\end{equation*}
\begin{equation*}
 \mu\langle V_1, \xi_1\rangle=O(\mu^2 \|X\|), \quad \mu\langle V_2, \xi_1\rangle=O(\mu \|X\|^2),
\end{equation*}
\begin{equation*}
\mu\langle B(X+V,X+V),\xi_1\rangle=O(\mu \|X\|^2),
\end{equation*}
\begin{equation*}
\langle  B(X,X),\xi_1\rangle=\langle  2B(A_1\xi_1,C\zeta_0),\xi_1\rangle= C A_1 k_1^2(1+\tau_1^2),
\end{equation*}
\begin{equation*}
\langle  2B(X-V_1-V_2,V_1),\xi_1\rangle=O(\mu \|X\|^2),
\end{equation*}
\begin{equation*}
\langle  2B(V_2,V_2),\xi_1\rangle=0,\quad \langle 2B(C\zeta_0,V_2),\xi_1\rangle=\langle 2B(A_1\xi_1,V_2),\xi_1\rangle=0,
\end{equation*}
\begin{equation*}
\langle 2B(\bar{A_1}\bar{\xi_1},V_2),\xi_1\rangle=A_1|A_1|^2 \left(\frac14 k_1^4(1+\tau_1^2)^2-\beta_0a_2k_1^2(1+\tau_1^2)^\frac32+\frac12 \beta_0a_1k_1^2(1+\tau_1^2) \right),
\end{equation*}
\begin{align*}
\langle 2B(A_2\xi_2,V_2),\xi_1\rangle=&A_1|A_2|^2\Big(\frac12k_1^2k_2^2(1-\tau_1\tau_2)^2-\frac12\beta_0 b_4k_1k_2(1-\tau_1\tau_2)+\frac12\beta_0 b_5k_1\\
&\cdot k_2(\sqrt{1+\tau_1^2}-\sqrt{1+\tau_2})-\frac12\beta_0b_6 k_1k_2(\tau_1\sqrt{1+\tau_2^2}+\tau_2\sqrt{1+\tau_1^2})\Big),
\end{align*}
\begin{align*}
\langle 2B(\bar{A}_2\bar{\xi}_2,V_2),\xi_1\rangle=&A_1|A_2|^2\Big(\frac12k_1^2k_2^2(1-\tau_1\tau_2)^2+\frac12\beta_0 b_1k_1k_2(1-\tau_1\tau_2)-\frac12\beta_0 b_2k_1\\
&\cdot k_2(\sqrt{1+\tau_1^2}+\sqrt{1+\tau_2})-\frac12\beta_0b_3 k_1k_2(\tau_1\sqrt{1+\tau_2^2}-\tau_2\sqrt{1+\tau_1^2})\Big).
\end{align*}
Therefore, one has a equation in the form
\begin{equation*}
A_1P_1(|A_1|^2,|A_2|^2,C,\mu)=0
\end{equation*}
where
\begin{align*}
P_1=\mu k_1^2(1+\tau_1^2)+C \beta_0 k_1^2(1+\tau_1^2)-\beta_0l_1 |A_1|^2&-\beta_0l_2 |A_2|^2\\&+O\{(|\mu|+|C|+|A_1|^2+|A_2|^2)^2\}
\end{align*}
with
\begin{equation*}
l_1=\frac14 k_1^4(1+\tau_1^2)^2-\beta_0a_2k_1^2(1+\tau_1^2)^\frac32+\frac12 \beta_0a_1k_1^2(1+\tau_1^2),
\end{equation*}
\begin{align*}
 l_2=&k_1^2k_2^2(1-\tau_1\tau_2)^2+\frac12\beta_0(b_1- b_4)k_1k_2(1-\tau_1\tau_2)-\frac12\beta_0 b_2k_1 k_2\Big(\sqrt{1+\tau_1^2}\\
&+\sqrt{1+\tau_2}\Big)-\frac12\beta_0b_3 k_1k_2\Big(\tau_1\sqrt{1+\tau_2^2}-\tau_2\sqrt{1+\tau_1^2}\Big)+\frac12\beta_0 b_5k_1 k_2\Big(\sqrt{1+\tau_1^2}\\
&-\sqrt{1+\tau_2}\Big)-\frac12\beta_0b_6 k_1k_2\Big(\tau_1\sqrt{1+\tau_2^2}+\tau_2\sqrt{1+\tau_1^2}\Big).
\end{align*}
Similarly,  substituting \eqref{u} into \eqref{e30}, we obtain
\begin{align*}
&\langle\mu U+(\beta_0+\mu)B(U,U),\xi_2\rangle\\
=&\mu\langle  X-V_1-V_2,\xi_2\rangle+(\beta_0+\mu)\langle B(X+V,X+V),\xi_2  \rangle=0,
\end{align*}
where
\begin{equation*}
\mu\langle X,\xi_2\rangle=\mu\langle A_2 \xi_2,\xi_2\rangle=A_2\mu k_2^2(1+\tau_2^2),
\end{equation*}
\begin{equation*}
 \mu\langle V_1, \xi_2\rangle=O(\mu^2 \|X\|), \quad \mu\langle V_2, \xi_2\rangle=O(\mu \|X\|^2),
\end{equation*}
\begin{equation*}
\mu\langle B(X+V,X+V),\xi_2\rangle=O(\mu \|X\|^2),
\end{equation*}
\begin{equation*}
\langle  B(X,X),\xi_2\rangle=\langle  2B(A_2\xi_2,C\zeta_0),\xi_2\rangle= C A_2 k_2^2(1+\tau_2^2),
\end{equation*}
\begin{equation*}
\langle  2B(X-V_1-V_2,V_1),\xi_2\rangle=O(\mu \|X\|^2),
\end{equation*}
\begin{equation*}
\langle  2B(V_2,V_2),\xi_2\rangle=0,\quad \langle 2B(C\zeta_0,V_2),\xi_2\rangle=\langle 2B(A_2\xi_2,V_2),\xi_2\rangle=0,
\end{equation*}
\begin{equation*}
\langle 2B(\bar{A_2}\bar{\xi_2},V_2),\xi_2\rangle=A_2|A_2|^2 \left(\frac12 k_2^4(1+\tau_2^2)^2-\beta_0a_4k_2^2(1+\tau_2^2)^\frac32+\frac12 \beta_0a_3k_2^2(1+\tau_2^2) \right),
\end{equation*}
\begin{align*}
\langle 2B(A_1\xi_1,V_2),\xi_2\rangle=&A_2|A_1|^2\Big(\frac12k_1^2k_2^2(1-\tau_1\tau_2)^2-\frac12\beta_0 b_4k_1k_2(1-\tau_1\tau_2)+\frac12\beta_0 b_5k_1\\
&\cdot k_2(\sqrt{1+\tau_1^2}-\sqrt{1+\tau_2})-\frac12\beta_0b_6 k_1k_2(\tau_1\sqrt{1+\tau_2^2}+\tau_2\sqrt{1+\tau_1^2})\Big),
\end{align*}
\begin{align*}
\langle 2B(\bar{A}_1\bar{\xi}_1,V_2),\xi_2\rangle=&A_2|A_1|^2\Big(\frac12k_1^2k_2^2(1-\tau_1\tau_2)^2+\frac12\beta_0 b_1k_1k_2(1-\tau_1\tau_2)-\frac12\beta_0 b_2k_1\\
&\cdot k_2(\sqrt{1+\tau_1^2}+\sqrt{1+\tau_2})-\frac12\beta_0b_3 k_1k_2(\tau_1\sqrt{1+\tau_2^2}-\tau_2\sqrt{1+\tau_1^2})\Big).
\end{align*}
It follows one equation in the form
\begin{equation*}
A_2P_2(|A_1|^2,|A_2|^2,C,\mu)=0
\end{equation*}
where
\begin{align*}
P_2=\mu k_2^2(1+\tau_2^2)+ C \beta_0k_2^2(1+\tau_2^2)-\beta_0l_3 |A_2|^2&-\beta_0l_2 |A_1|^2\\&+O((|\mu|+|C|+|A_1|^2+|A_2|^2)^2)
\end{align*}
with
\begin{equation*}
l_3=\frac14 k_2^4(1+\tau_2^2)^2-\beta_0a_4k_2^2(1+\tau_2^2)^\frac32+\frac12 \beta_0a_3k_2^2(1+\tau_2^2).
\end{equation*}
Applying Cramer's rule will lead to (dropping the term of $O((|\mu|+|C|+|A_1|^2+|A_2|^2)^2)$):
\begin{align}
|A_1|^2=&\frac{(\mu+C\beta_0)\left[l_3k_1^2(1+\tau_1^2)-l_2k_2^2(1+\tau_2^2)\right]}{\beta_0(l_1l_3-l_2^2)},\label{A1}\\
|A_2|^2=&\frac{(\mu+C\beta_0)\left[l_1k_2^2(1+\tau_2^2)-l_2k_1^2(1+\tau_1^2)\right]}{\beta_0(l_1l_3-l_2^2)},\label{A2}
\end{align}
provided that the followings are satisfied
\begin{equation}\label{B}
l_1l_3-l^2_2\neq 0,
\end{equation}
\begin{equation}\label{C}
\frac{(\mu+C\beta_0)\left[l_3k_1^2(1+\tau_1^2)-l_2k_2^2(1+\tau_2^2)\right]}{\beta_0(l_1l_3-l_2^2)}\geq0,
\end{equation}
\begin{equation}\label{D}
\frac{(\mu+C\beta_0)\left[l_1k_2^2(1+\tau_2^2)-l_2k_1^2(1+\tau_1^2)\right]}{\beta_0(l_1l_3-l_2^2)}\geq 0.
\end{equation}
Therefore,
\begin{align*}
U=&C \zeta_0+A_1 \xi_1+\bar{A}_1\bar{\xi}_1+A_2\xi_2+\bar{A}_2\bar{\xi}_2-V_1(\mu,C,A_1,A_2)-V_2(\mu,C,A_1,A_2)\\
&+O((|\mu|+|C|+|A_2|^2+|A_1|^2)^2)
\end{align*}
where $C$  are respectively the averages of the elevation $\eta$. In summary, we provide the following theorem.

\begin{theorem}\label{thm1}
For  positive $\Omega$, $k_1$, $k_2$, $\tau_1$ and $\tau_2$, we consider $\beta_0>0$  such that
\begin{equation*}
   \Omega_{\beta_0}:=\{(n_1,n_2,q)\in \Z^2 \times \Z^+ ,q\left(1+\omega (p_1^2+p_2^2)\right)-\beta_0 \sqrt{p_1^2+p_2^2}=0\}
\end{equation*}
 only contains elements $(\pm 1,0,q)$ and $(0,\pm1,q)$ with  $q\in \Z$. Then, for  $\mu$, $C$ close enough to 0 and satisfying conditions (\ref{B}) (\ref{C}) and (\ref{D}) where $C$ is the average of the evolution $\eta$, there exists a family of bifurcation standing waves in the form $$U=U(\mathbf{x}+\vec{\sigma},t)\in H^{s}_{\natural,e}\times H^{s}_{\natural,o} \times H^{s}_{\natural,o},$$
 for $s\geq 2$ with $\vec{\sigma}$ being a space shift, where
\begin{equation*}
U(\mathbf{x},t)=\left(\eta^0(\mathbf{x},t),\mathbf{v}^0(\mathbf{x},t)\right),
\end{equation*}
with
\begin{align*}
\eta^0(\mathbf{x},t)=&C+2k_1\sqrt{1+\tau_1^2}|A_1|\cos{(k_1x_1+k_1\tau_1 x_2)}\cos{qt}+2k_2\sqrt{1+\tau_2^2}|A_2|\\
&\cdot \cos{(k_2x_2-k_2\tau_2 x_2)}\cos{qt}+V^{(1)}(\mu, C, A_1,A_2),
\end{align*}
\begin{align*}
\mathbf{v}^0_1(\mathbf{x},t)=&-2k_1 |A_1|\sin{(k_1x_1+k_1\tau_1 x_2)}\sin{qt}-2k_2|A_2|\sin{(k_2x_1-k_2\tau_2 x_2)}\\
&\cdot\sin{qt}+V^{(2)}(\mu, C, A_1,A_2),
\end{align*}
\begin{align*}
\mathbf{v}^0_2(\mathbf{x},t)=&-2k_1\tau_1 |A_1|\sin{(k_1x_1+k_1\tau_1 x_2)}\sin{qt}+2k_2\tau_2 |A_2|\sin{(k_2x_1-k_2\tau_2 x_2)}\\
&\cdot\sin{qt}
+V^{(3)}(\mu, C, A_1,A_2).
\end{align*}
for $|A_1|$ and $|A_2|$ are shown in (\ref{A1}) and (\ref{A2}) and
\begin{equation*}
    V(\mu, C, A_1, A_2)= (V^{(1)},V^{(2)},V^{(3)})=-V_1-V_2+O((|\mu|+|C|+|A_1|^2+|A_2|^2)^2)
\end{equation*}
with $V_1$ and $V_2$ given in (\ref{V12}).
\end{theorem}
Furthermore, we are also able to calculate higher order terms according to \eqref{u} and \eqref{V12}.
\begin{corollary}\label{e2}
With the assumption of the theorem above, the form of the free surface, even in time $t$, with $C=0$ is given by $\eta=\eta(\mathbf{x}-\vec{\sigma},t) $ with
\begin{align*}
&\eta(x_1,x_2,t)\\
=& 2k_1\sqrt{1+\tau_1^2} |A_1|
\cos{(k_1x_1+k_1\tau_1 x_2)}\cos{qt}+2k_2\sqrt{1+\tau_2^2} |A_2|\\
&\cdot\cos{(k_2x_1-k_2\tau_2x_2)}\cos{qt}+\Big[\frac12 k_1^2(1+\tau_1^2)|A_1|^2 \cos{(2k_1x_1+2k_1\tau_1x_2)}\\
&+\frac12 k_2^2(1+\tau_2^2)|A_2|^2\cos{(2k_2x_1-2k_2\tau_2x_2)}\\
&-k_1k_2(1-\tau_1\tau_2)|A_1A_2|\cos{\big((k_1+k_2)x_1+(k_1\tau_1-k_2\tau_2)x_2\big)}\\
&-k_1k_2(1-\tau_1\tau_2)|A_1A_2|\cos{\big((k_1-k_2)x_1+(k_1\tau_1+k_2\tau_2)x_2\big)}\Big]\\
&-2\beta_0a_1|A_1|^2 \cos{(2k_1x_1+2k_1\tau_1x_2)}\cos{2qt}-2\beta_0a_3|A_2|^2 \\
&\cdot\cos{(2k_2x_1-2k_2\tau_2x_2)}\cos{2qt}-2\beta_0b_1|A_1A_2| \cos \big((k_1+k_2)x_1\\
&+(k_1\tau_1-k_2\tau_2)x_2\big)\cos{2qt}-2\beta_0b_4|A_1A_2| \cos\big((k_1-k_2)x_1\\
&+(k_1\tau_1+k_2\tau_2)x_2)\big)\cos{2qt}+O(|\mu|^\frac32),
\end{align*}
for $\mu$ close to zero and  $A_1$ and $A_2$ being given in (\ref{A1})-(\ref{A2}).
\end{corollary}

 We can now plot the standing waves surface in the $(x_1, x_2)$ plane presented in Theorem \ref{thm1} and Corollary \ref{e2}. For fixed $(\omega,k_1,k_2,\tau_1,\tau_2)$, we choose $\beta_0$ such that $(\pm1, 0, 1)$ and $(0,\pm 1, 1)$  are the only four elements in $\Omega_{\beta_0}$ (See Lemma \ref{U}), that is,
 \[(1+\omega(k_j^2+k_j^2\tau_j^2))-\beta_0 \sqrt{k_j^2+k_j^2\tau_j^2}=0,\quad \mbox{for } j=1,2.   \]
  When $\tau_1=\tau_2$, $\Gamma$ is a symmetric diamond lattice. Figure 1 ,2, 3 and 4 show the standing waves in such symmetric lattice. When $\tau_1\neq \tau_2$, $\Gamma$ is an asymmetric hexagonal lattice. Figure  5 shows the standing waves in such lattice.

\begin{figure}[!h]
\centering
   \includegraphics[width=13
cm]{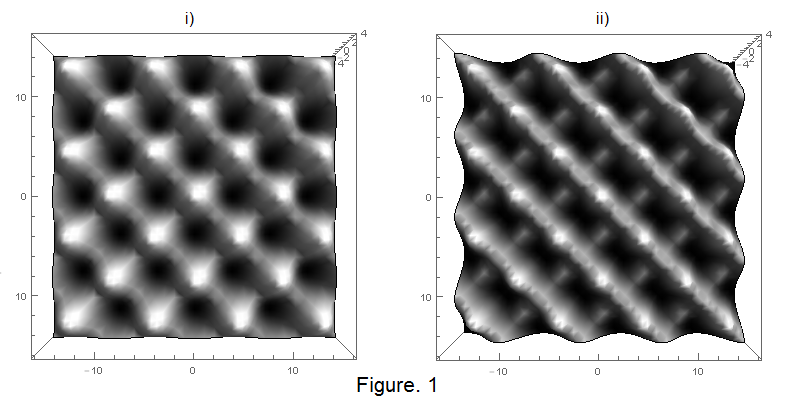}\\
\end{figure}
 For \[(\omega,k_1,k_2,\tau_1,\tau_2)=(1,\frac{1}{\sqrt{2}},\frac{1}{\sqrt{2}},1,1)\] with $(\beta_0, C, \mu)=(2,0,-0.3)$, one has $|A_1|=|A_2|=0.539972$. Here, $\tau_1=\tau_2=1$ and $k_1=k_2$ give us a square shape lattice $\Gamma$. Figure 1 shows symmetric square shape standing waves at $t=\pi/20$ in the exact same coordinates. i) is the standing waves with the leading order, ii) is the waves including the second order as in Corollary \ref{e2}.

\begin{figure}[!h]
\centering
   \includegraphics[width=13
cm]{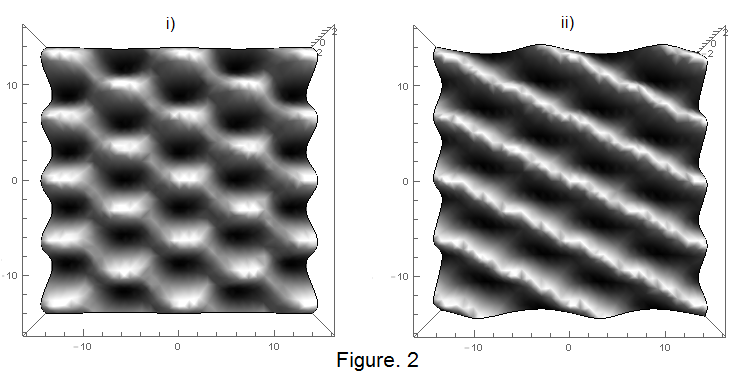}\\
\end{figure}
 For \[(\omega,k_1,k_2,\tau_1,\tau_2)=(1,\frac{1}{\sqrt{5}},\frac{1}{\sqrt{5}},2,2)\]
 with $(\beta_0, C, \mu)=(2,0,-0.02)$, one has $|A_1|=|A_2|=0.397342$. Here, $\tau_1=\tau_2=2$ and $k_1=k_2$ give us a diamond shape lattice $\Gamma$. Figure 2 shows symmetric diamond shape standing waves at $t=\pi/20$ in the exact same coordinates. i) is the standing waves with the leading order, ii) is the waves including the second order.

\begin{figure}[!h]
\centering
   \includegraphics[width=13.5
cm]{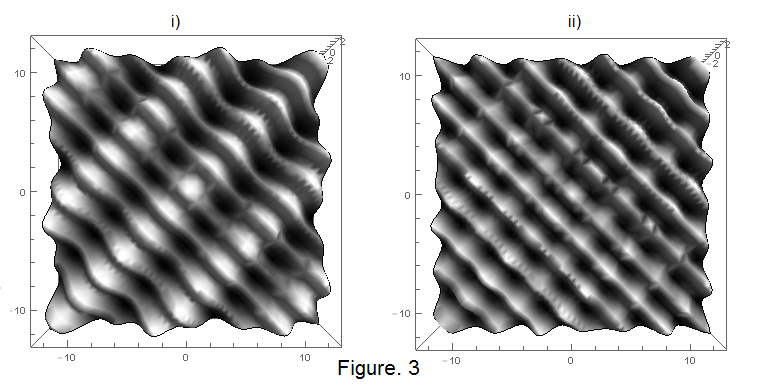}\\
\end{figure}
 For \[(\omega,k_1,k_2,\tau_1,\tau_2)=(1,\frac{\sqrt{5}+1}{2\sqrt{2}},\frac{\sqrt{5}-1}{2\sqrt{2}},1,1)\] with $(\beta_0, C, \mu)=(\sqrt{5},0,-0.2)$, one has $|A_1|=0.18089$ and $|A_2|=0.784643$. Here, $\tau_1=\tau_2=1$ and $k_1\neq k_2$ give us a rectangular shape lattice $\Gamma$. Figure 3 shows  symmetric rectangular standing waves at $t=\pi/20$ in the exact same coordinates. i) is the standing waves with the leading order, ii) is the waves including the second order.

\begin{figure}[!h]
\centering
   \includegraphics[width=13
cm]{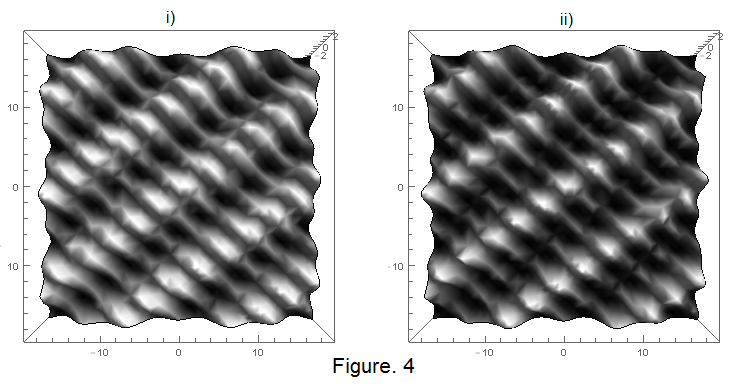}\\
\end{figure}

\begin{figure}[!h]
\centering
   \includegraphics[width=13
cm]{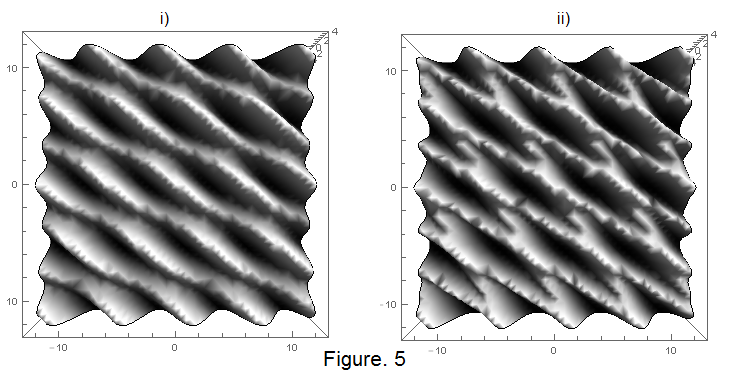}\\
\end{figure}

\ For \[(\omega,k_1,k_2,\tau_1,\tau_2)=(1,\frac{\sqrt{5}+1}{2\sqrt{3}},\frac{\sqrt{5}-1}{2\sqrt{3}},\sqrt{2},\sqrt{2})\] with $(\beta_0, C, \mu)=(\sqrt{5},0,-0.08)$, one has $|A_1|=0.186587$ and $|A_2|=0.730987$. Here, $\tau_1=\tau_2=2$ provides a symmetric lattice with $k_1\neq k_2$.  Figure 4 shows a asymmetric parallelogram standing waves at $t=\pi/20$ in the exact same coordinates. i) shows the standing waves with the leading order, ii) shows the waves including the second order.\\

 For \[(\omega,k_1,k_2,\tau_1,\tau_2)=(1,1,1,\frac{(\sqrt{5}+1)^2}{4}-1,\frac{(\sqrt{5}-1)^2}{4}-1)\] 
 with $(\beta_0, C, \mu)=(\sqrt{5},0,0.5)$, one has $|A_1|=0.47833$ and $|A_2|=0.620853$. Here,  $\tau_1\neq \tau_2$ gives an asymmetric lattice. Figure 5 shows asymmetric parallelogram standing waves at $t=\pi/20$ in the exact same coordinates. i) shows the standing waves with the leading order, ii) shows the waves including the second order.


%
%
%
%
%
%
%

\end{document}